\documentclass[11pt,a4paper]{article}

\usepackage{amssymb,amsmath,amsfonts,amsthm,color,mathrsfs}
\usepackage{bbm}
\usepackage[marginal]{footmisc}

\title{\bf Li--Yau Inequalities for Dunkl Heat Equations}
\author{Huaiqian Li\footnote{Email: {\color{blue}huaiqianlee@gmail.com}}
\quad Bin Qian\footnote{Email: {\color{blue} binqiancs@yahoo.com   }}
 \vspace{2mm}
\\
{\footnotesize Center for Applied Mathematics, Tianjin University, Tianjin 300072, P. R. China} \\
{\footnotesize Department of Mathematics and Statistics, Changshu Institute of Technology,} \\
{\footnotesize Changshu, Jiangsu 215500, P. R. China.}
}

\date{}

\usepackage{color} 

\def\R{\mathbb{R}}

\def\D{\mathbb{D}}

\def\d{\textup{d}}
\def\D{\textup{D}}

\def\<{\langle}
\def\>{\rangle}
\def\Proof.{\noindent{\bf Proof. }}

\def\newdot{{\kern.8pt\cdot\kern.8pt}}

\newtheorem{theorem}{Theorem}[section]
\newtheorem{lemma}[theorem]{Lemma}
\newtheorem{corollary}[theorem]{Corollary}

\theoremstyle{definition}\newtheorem{remark}[theorem]{Remark}

\begin{document}
\allowdisplaybreaks
\maketitle
\makeatletter 
\renewcommand\theequation{\thesection.\arabic{equation}}
\@addtoreset{equation}{section}
\makeatother 

\begin{abstract}
Motivated by recent works due to Yu--Zhao [J. Geom. Anal. 2020] and Weber--Zacher [arXiv:2012.12974], we study Li--Yau inequalities for the heat equation corresponding to the Dunkl Laplacian, which is a non-local operator parameterized by reflection groups and multiplicity functions. The results are sharp in the particular case when the reflection group is isometric to $\mathbb{Z}_2^d$.
\end{abstract}



\section{Introduction and main results}\hskip\parindent
Let $M$ be a $d$-dimensional complete Riemannian manifold with non-negative Ricci curvature. The pointwise Li--Yau inequality initially obtained in the seminal paper \cite{LiYau86} states that, for every positive solution to the heat equation $\partial_t u=\Delta u$ on $(0,\infty)\times M$, it holds
\begin{eqnarray}\label{LY86}
-\Delta\big(\log u(t,\cdot)\big)(x)\leq\frac{d}{2t},\quad t>0,\, x\in M,
\end{eqnarray} which is equivalent to
\begin{eqnarray}\label{LY86+}
\frac{|\nabla u(t,\cdot)|^2(x)}{u(t,x)^2}-\frac{\partial_t u(t,x)}{u(t,x)}\leq\frac{d}{2t},\quad t>0,\, x\in M,
\end{eqnarray}
where $\Delta$ is the Laplace--Beltrami operator, $\nabla$ is the Riemannian gradient and $|\cdot|$ is the length in the tangent space. Let $\rho(\cdot,\cdot)$ be the Riemannian distance on $M$. \eqref{LY86} immediate implies the Li--Yau parabolic Harnack inequality, i.e.,
$$u(s,x)\leq u(t,y)\Big(\frac{t}{s}\Big)^{d/2}\exp\Big(\frac{\rho(x,y)^2}{4(t-s)}\Big),\quad x,y\in M,\,0<s<t<\infty,$$
and  is crucial to obtain estimates on the corresponding heat kernel and its gradient. Note that \eqref{LY86} is sharp in the sense that equality is achieved for the fundamental solution of the heat equation on the $d$-dimensional Euclidean space $\R^d$ (see \eqref{Gauss-kernel} below). Up to now, there are quite a few works on improving this inequality for small time and large time; see e.g. \cite{Hamilton93,Li-Xu11,BBG2017,WZ2021} and references therein.

Recently, in \cite{YZ2020}, sharp Li--Yau inequalities for the Laplace--Beltrami operator on hyperbolic spaces were obtained by employing the explicit formula for the corresponding heat kernel. Very recently, in \cite{WZ2021}, similar to the idea of \cite{YZ2020}, the Li--Yau inequality in the sense of \eqref{LY86} for the fractional Laplacian has been proved; however, the Li--Yau inequality of gradient type in the sense of \eqref{LY86+} has not been mentioned, where the ``gradient'' should be understood as the the carr\'{e} du champ operator induced by the fractional Laplacian (see also the conjectures at the end of \cite[Section 21]{Garo2018}). We should mention that there are works on the Li--Yau inequality in the setting of graphs via various curvature-dimension conditions in the sense of Barky--Emery \cite{BE1985}; see e.g. \cite{BHLLMY,Qian2017,Munch2018,DKZ} and references there in.

Motivated by \cite{YZ2020} and \cite{WZ2021}, in the present paper, we mainly study Li--Yau inequalities for the generalized heat equation and the heat kernel corresponding to the Dunkl Laplacian, which is a non-local operator parameterized by reflection groups and multiplicity functions.

Let $\langle\cdot,\cdot\rangle$ be the standard scalar product and $|\cdot|$ be the associated norm on $\R^d$. For every $\alpha\in\R^d\setminus\{0\}$, let $\alpha^\bot:=\{x\in\R^d: \langle\alpha,x\rangle=0\}$, the hyperplane orthogonal to $\alpha$, and denote $r_\alpha$ the mirror reflection with respect to $\alpha^\bot$, i.e.,
$$r_\alpha x =x-2\frac{\langle \alpha,x\rangle}{|\alpha|^2}\alpha,\quad x\in\R^d.$$

Let $\mathfrak{R}$ be a root system on $\R^d$, which means that $\mathfrak{R}$ is a finite set of nonzero vectors in $\R^d$ such that, for every $\alpha\in\mathfrak{R}$, $r_\alpha(\mathfrak{R})=\mathfrak{R}$ and $\alpha\R\cap \mathfrak{R}=\{\alpha,-\alpha\}$, where $\alpha\R:=\{a\alpha: a\in\R\}$. Let $\mathfrak{R}_+$ be a positive subsystem of $\mathfrak{R}$ such that $\mathfrak{R}$ can be written as the disjoint union of $\mathfrak{R}_+$ and $-\mathfrak{R}_+$, where $-\mathfrak{R}_+:=\{-\alpha: \alpha\in\mathfrak{R}_+\}$. The root system $\mathfrak{R}$ generates the reflection group, denoted by $G$, which is known to be a  finite  subgroup of the orthogonal group $O(d)$. Let $\kappa: \mathfrak{R}\rightarrow\R_+$ be the multiplicity function such that $\kappa_{g\alpha}=\kappa_\alpha$ for every $g\in G$ and every $\alpha\in\mathfrak{R}$. Without loss of generality, we normalize that $|\alpha|^2=2$ for every $\alpha\in\mathfrak{R}$.

Introduced in \cite{Dunkl1989} by C.F. Dunkl, for every $\xi\in\R^d$, the Dunkl operator $\D_\xi$ along $\xi$ associated to $\mathfrak{R}$ and $\kappa$ is defined by
$$\D_\xi f=\partial_\xi f+\sum_{\alpha\in\mathfrak{R}_+}\kappa_\alpha \langle\alpha,\xi\rangle \frac{f-r_\alpha f}{\alpha^*},\quad f\in C^1(\R^d),$$
where $\partial_\xi$ denotes the directional derivative along $\xi$ and, for every $x\in\R^d$ and every $\alpha\in\mathfrak{R}$, $\alpha^*(x):=\langle\alpha,x\rangle$, $r_\alpha f(x):=f(r_\alpha x)$. For convenience, write $\D_i$ for $\D_{e_i}$ and $\partial_i$ for $\partial_{e_i}$, $i=1,\cdots,d$, where $\{e_i\}_{i=1}^d$ is the standard orthonormal basis of $\R^d$. Denote $\nabla_\kappa=(\D_1,\cdots,\D_d)$ the Dunkl gradient and
$\Delta_\kappa=\sum_{i=1}^d \D_i^2$ the Dunkl Laplacian. From direct calculation, it is easy to see that
$$\Delta_\kappa f=\Delta f+2\sum_{\alpha\in\mathfrak{R}_+}\kappa_\alpha\Big[\frac{\langle\alpha,\nabla f\rangle}{\alpha^*} - \frac{f-r_\alpha f}{(\alpha^*)^2}\Big],\quad f\in C^2(\R^d).$$
In particular, when $\kappa=0$, $\nabla_0=\nabla$ and $\Delta_0=\Delta$ are the standard gradient operator and the standard Laplacian on $\R^d$, respectively.

Define the natural weight function $\omega_\kappa$ by
$$\omega_\kappa(x):=\prod_{\alpha\in\mathfrak{R}_+}|\langle \alpha,x \rangle|^{\kappa_\alpha},\quad x\in\R^d,$$
which is a homogeneous function of degree
$$\lambda_\kappa:=\sum_{\alpha\in\mathfrak{R}_+}\kappa_\alpha.$$
Let $\mu_{\kappa}(\d x)=\omega_\kappa(x)\d x$, where $\d x$ denotes the Lebesgue measure on $\R^d$.

Let $(p_t)_{t>0}$ be the Dunkl heat kernel corresponding to the Dunkl heat semigroup $(P_t)_{t>0}$, where $P_t:=e^{t\Delta_\kappa}$ for every $t>0$. See Section 3 and \cite{Rosler2003,ADH2019} for more information for the Dunkl heat semigroup and estimates on the Dunkl heat kernel.

The main results of the present paper are in order. The first one is on the Li--Yau inequality for the Dunkl heat kernel in the special situation when the reflection group $G$ is isomorphic to the Abelian group $\mathbb{Z}_2^d$.
\begin{theorem}\label{main-thm-1}
Suppose that $G$ is isomorphic to $\mathbb{Z}_2^d$. Then, for every $t>0$ and every $x,y\in\R$, the Li--Yau inequality for the Dunkl heat kernel holds, i.e.,
\begin{eqnarray}\label{LY-1}
-\Delta_\kappa\big(\log p_t(\cdot,y)\big)(x)\leq \frac{d+2\lambda_\kappa}{2t}.
\end{eqnarray}
\end{theorem}
We remark that $d+2\lambda_\kappa$ should be regarded as the homogeneous dimension of the Dunkl system. We should emphasize that  \eqref{LY-1} is sharp in the sense that, if $\kappa=0$, then $\Delta_\kappa=\Delta$, $\lambda_\kappa=0$, and $(p_t)_{t>0}$ is the heat kernel corresponding to $\Delta$ on $\R^d$, i.e.,
\begin{eqnarray}\label{Gauss-kernel}
p_t(x,y)=\frac{1}{(4\pi t)^{d/2}}\exp\Big(-\frac{|x-y|^2}{4t}\Big),\quad x,y\in\R^d,\,t>0;
\end{eqnarray}
hence
$$-\Delta\big(\log p_t(\cdot,y)\big)(x)=\frac{d}{2t},$$
for every $t>0$ and every $x,y\in\R^d$.

The results in Theorem \ref{main-thm-1} and in \cite[Theorem 3.2]{WZ2021} are not comparable, since from the probabilistic point of view, the stochastic process (also called Dunkl process) corresponding to the Dunkl Laplacian $\Delta_\kappa$ is a Markov jump process but not a L\'{e}vy process if $\kappa>0$ (see \cite[Section 3]{LZ2020} and also \cite{GaYor2005} for the explicit expression of the jumping kernel), while the process corresponding to the fractional Laplacian $(-\Delta)^{\beta/2}$ with $\beta\in(0,2)$ is an isotropic $\beta$-stable L\'{e}vy process (see e.g. \cite{Sato1999} for more details on L\'{e}vy processes).

The second one is on Li--Yau inequalities for solutions of the Dunkl heat equation.
\begin{theorem}\label{main-thm-2}
Let $T\in(0,\infty]$ and $\beta: (0,T)\times \R^d\rightarrow\R$ be a function.  Suppose that $u:[0,T)\times \R\rightarrow (0,\infty)$ is any $C^2$ solution to the Dunkl heat equation
\begin{equation}\label{dunkl-heat-equ}
\partial_t u(t,x)=\Delta_\kappa\big(u(t,\cdot)\big)(x),\quad (t,x)\in (0,T)\times\R^d.
\end{equation}
Then
\begin{eqnarray}\label{thm-2-0}
-\Delta_\kappa\big(\log p_t(\cdot,y)\big)(x)\leq \beta(t,x), \quad (t,x,y)\in (0,T)\times \R^d\times \R^d,
\end{eqnarray}
is equivalent to
\begin{eqnarray}\label{thm-2-1}
-\Delta_\kappa\big(\log u(t,\cdot)\big)(x)\leq \beta(t,x), \quad (t,x)\in (0,T)\times \R^d;
\end{eqnarray}
moreover, either \eqref{thm-2-0} or \eqref{thm-2-1} implies
\begin{eqnarray}\label{thm-2-1+}
\frac{|\nabla u(t,\cdot)(x)|^2}{u(t,x)^2}-\frac{\partial_t u(t,x)}{u(t,x)}\leq  \beta(t,x), \quad (t,x)\in (0,T)\times \R^d.
\end{eqnarray}
In addition, if $G$ is isomorphic to $\mathbb{Z}_2^d$, then
\begin{eqnarray}\label{thm-2-2}
-\Delta_\kappa\big(\log u(t,\cdot)\big)(x)\leq \frac{d+2\lambda_\kappa}{2t},\quad (t,x)\in (0,T)\times\R^d.
\end{eqnarray}
\end{theorem}

An important direct consequence of Theorem \ref{main-thm-2}, presented in the next corollary, is the Harnack inequality, which is first established by B. Pini \cite{Pini1954} and J. Hadamard \cite{Had1954} for the standard heat equation on the Euclidean plane and then derived by P. Li and S.-T. Yau in \cite{LiYau86} from the Li--Yau inequality on Riemannian manifolds.
\begin{corollary}\label{harnack}
Let $T\in(0,\infty]$. Suppose that $u:[0,T)\times \R\rightarrow (0,\infty)$ is a $C^2$ solution to the Dunkl heat equation \eqref{dunkl-heat-equ} and the reflection group $G$ is isomorphic to $\mathbb{Z}_2^d$. Then for every $0<s<t<T$ and every $x,y\in\R^d$,
\begin{eqnarray*}
u(s,x)\leq u(t,y)\Big(\frac{t}{s}\Big)^{\lambda_\kappa+d/2}\exp\Big(\frac{|x-y|^2}{4(t-s)}\Big).
\end{eqnarray*}
\end{corollary}

For more details on the Dunkl theory, refer to the survey paper \cite{Rosler2003} and the books \cite{DunklXu2014,DX2015}. In the remainder of the paper, the proof of Theorem \ref{main-thm-1} is presented in Section 2, and proofs of Theorem \ref{main-thm-2} and Corollary \ref{harnack} are given in Section 3.

\section{Proofs of Theorem \ref{main-thm-1}}\hskip\parindent
In this section, we shall consider the special situation when the reflection group $G$ is isomorphic to $\mathbb{Z}_2^d$. In this case, given real numbers $\kappa_i\geq0$, $i=1,\cdots,d$, for every $x=(x_1,\cdots,x_d)\in\R^d$,
$$\D_i f(x)=\partial_i f(x)+\frac{\kappa_i}{x_i}\big[f(x)-f(r_i x)\big],\quad f\in C^1(\R^d),\,i=1,\cdots,d,$$
and
\begin{eqnarray*}
\Delta_\kappa f(x)&=&\Delta f(x)+\sum_{i=1}^d \frac{\kappa_i}{x_i^2}\big[2x_i\partial_i f(x)-f(x)+f(r_i x)\big]\cr
&=&\sum_{i=1}^d\Big(\partial_{ii}^2 f(x) + \frac{\kappa_i}{x_i^2}\big[2x_i\partial_i f(x)-f(x)+f(r_i x)\big]\Big),\quad f\in C^2(\R^d),
\end{eqnarray*}
where $r_i(x_1,\cdots,x_i,\cdots,x_d)=(x_1,\cdots,x_{i-1},-x_i,x_{i+1},\cdots,x_d)$, $i=1,\cdots,d$. The weight function has the expression
$$\omega_\kappa(x)=\prod_{i=1}^d|x_i|^{\kappa_i},\quad x\in\R^d,$$
which is clearly homogeneous of degree $\lambda_\kappa=\sum_{i=1}^d \kappa_i$.

Moreover, for every $t>0$ and every $x,y\in\R^d$ with $x=(x_1,\cdots,x_d),y=(y_1,\cdots,y_d)$, $p_t(x,y)$ has the following explicit expression (see e.g. \cite{Dunkl1992,Rosler1998}), i.e.,
$$p_t(x,y)=\frac{1}{c_\kappa (2t)^{\lambda_\kappa+d/2}}\exp\Big(-\frac{|x|^2+|y|^2}{4t}\Big)\prod_{i=1}^d E_{\kappa_i}\Big(\frac{x_i}{\sqrt{2t}},\frac{y_i}{\sqrt{2t}}\Big),$$
where
\begin{eqnarray*}
c_\kappa&:=&\int_{\R^d}e^{-|x|^2/2}\,\mu_\kappa(\d x)=\prod_{i=1}^d\Gamma\Big(\kappa_i+\frac{1}{2}\Big),\\
E_{\kappa_i}(x_i,y_i)&:=&\frac{\Gamma(\kappa_i+1/2)}{\sqrt{\pi} \Gamma(\kappa_i)}\int_{-1}^1 (1-s)^{\kappa_i-1}(1+s)^{\kappa_i} e^{sx_i y_i}\,\d s,\quad i=1,\cdots,d,
\end{eqnarray*}
and $\Gamma(\cdot)$ is the Gamma function. For convenience, for each $i=1,\cdots,d$, let $c_{\kappa_i}=\Gamma(\kappa_i+1/2)$ and
\begin{eqnarray}\label{1d-kernel}
p^i_t(u,v)=\frac{1}{c_{\kappa_i} (2t)^{\kappa_i+1/2}}\exp\Big(-\frac{u^2+v^2}{4t}\Big)
E_{\kappa_i}\Big(\frac{u}{\sqrt{2t}},\frac{v}{\sqrt{2t}}\Big),\quad u,v\in\R,\,t>0.
\end{eqnarray}
Then, for every $t>0$ and every $x,y\in\R^d$ with $x=(x_1,\cdots,x_d)$ and $y=(y_1,\cdots,y_d)$,
\begin{eqnarray}\label{product-kernel}
p_t(x,y)=\prod_{i=1}^dp^i_t(x_i,y_i).
\end{eqnarray}
In the $\mathbb{Z}_2^d$ setting, refer to  \cite{ABDH2015} for more details on the Dunkl heat kernel and its estimates and to \cite{DX2015} for related materials.

Now we are ready to present the proof.
\begin{proof}[Proof of Theorem \ref{main-thm-1}] Let $t>0$ and $x,y\in\R^d$ with $x=(x_1,\cdots,x_d),y=(y_1,\cdots,y_d)$. By \eqref{product-kernel}, we immediately have
\begin{eqnarray}\label{thm-1-proof-1}
&&\Delta_\kappa\big(\log p_t(\cdot,y)\big)(x)=\sum_{j=1}^d \Delta_\kappa\big(\log p_t^j(\cdot,y_j)\big)(x_j)\cr
&=&\sum_{j=1}^d \Big(\frac{\kappa_j}{x_j^2}\big[2x_j\partial_{x_j}\log p_t^j(x_j,y_j)
-\log p_t^j(x_j,y_j)+ \log p_t^j(-x_j,y_j)\big]\cr
&&+\partial_{x_jx_j}^2 \log p_t^j(x_j,y_j)\Big).
\end{eqnarray}
Hence, it suffices to estimate the terms in the parentheses of \eqref{thm-1-proof-1}, i.e.,
$$\partial_{x_jx_j}^2 \log p_t^j(x_j,y_j)+\frac{\kappa_j}{x_j^2}\big[2x_j\partial_{x_j}\log p_t^j(x_j,y_j)
-\log p_t^j(x_j,y_j)+ \log p_t^j(-x_j,y_j)\big]=:{\rm I}_j,$$
and it reduces to the rank-one case.

For notational convenience, we ignore the subscript $j$ of ${\rm I}_j$ and for every $t>0$ and $x,y\in\R$, we let
\begin{eqnarray*}
{\rm I}:=\partial_{xx}^2 \log p_t(x,y)&+&\frac{\kappa}{x^2}\big[2x\partial_{x}\log p_t(x,y)-\log p_t(x,y)+ \log p_t(-x,y)\big],
\end{eqnarray*}
where $p_t(x,y)$ is given in \eqref{1d-kernel} by ignoring the subscript $i$.

Let $C_{\kappa}=\frac{\Gamma(\kappa+1/2)}{\sqrt{\pi}\Gamma(\kappa)}$ and  $g(s)=(1-s)^{\kappa-1}(1+s)^{\kappa}$. Then
$E_{\kappa}(x,y)=C_{\kappa}\int_{-1}^1e^{sxy}g(s)\,\d s$. Hence
\begin{eqnarray*}
\partial_x E_{\kappa}\Big(\frac{x}{\sqrt{2t}},\frac{y}{\sqrt{2t}}\Big)&=&C_{\kappa}\int_{-1}^1\frac{sy}{2t}e^{\frac{sxy}{2t}}g(s)\,\d s,\\
\partial_{xx}^2 E_\kappa\Big(\frac{x}{\sqrt{2t}},\frac{y}{\sqrt{2t}}\Big)&=&C_\kappa\int_{-1}^1\Big(\frac{sy}{2t}\Big)^2e^{\frac{sxy}{2t}}g(s)\,\d s.
\end{eqnarray*}

From \eqref{1d-kernel}, for every $t>0$ and $x,y\in\R$,
$$\log p_t(x,y)=-\log c_\kappa-\Big(\kappa+\frac{1}{2}\Big)\log(2 t)-\frac{x^2+y^2}{4t}+\log E_\kappa\Big(\frac{x}{\sqrt{2t}},\frac{y}{\sqrt{2t}}\Big),$$
and hence
\begin{eqnarray*}
\partial_{xx}^2\log p_t(x,y)&=&-\frac{1}{2t}+\frac{\partial_{xx}^2 E_\kappa\Big(\frac{x}{\sqrt{2t}},\frac{y}{\sqrt{2t}}\Big)}{ E_\kappa\Big(\frac{x}{\sqrt{2t}},\frac{y}{\sqrt{2t}}\Big)}-\frac{\bigg(\partial_x E_\kappa\Big(\frac{x}{\sqrt{2t}},\frac{y}{\sqrt{2t}}\Big)\bigg)^2}{E_\kappa\Big(\frac{x}{\sqrt{2t}},\frac{y}{\sqrt{2t}}\Big)^2}\\
&=:&-\frac{1}{2t}+A-B.
\end{eqnarray*}
It is easy to deduce that
\begin{eqnarray*}
A-B
&=&\frac{y^2}{4t^2}\left(\frac{\int_{-1}^1 s^2 g(s)e^{\frac{sxy}{2t}}\,\d s}{\int_{-1}^1 g(s)e^{\frac{sxy}{2t}}\,\d s}  -
\frac{\big(\int_{-1}^1 s g(s)e^{\frac{sxy}{2t}}\,\d s\big)^2}{\big(\int_{-1}^1 g(s)e^{\frac{sxy}{2t}}\,\d s\big)^2}\right)\\
&=&\frac{y^2}{4t^2}\frac{\big(\int_{-1}^1 s^2 g(s)e^{\frac{sxy}{2t}}\,\d s\big)\big(\int_{-1}^1 g(s)e^{\frac{sxy}{2t}}\,\d s\big)  -  \big(\int_{-1}^1 s g(s)e^{\frac{sxy}{2t}}\,\d s\big)^2}{\big(\int_{-1}^1 g(s)e^{\frac{sxy}{2t}}\,\d s\big)^2}\\
&\geq&0,
\end{eqnarray*}
by the Cauchy--Schwarz inequality. 
Hence
\begin{eqnarray}\label{proof-thm-1-1}
\partial_{xx}^2\log p_t(x,y)\geq-\frac{1}{2t}.
\end{eqnarray}

Let
$${\rm J}=\frac{\kappa}{x^2}\big[2x\partial_x\log p_t(x,y)+\log p_t(-x,y)-\log p_t(x,y)\big].$$
Then, letting $a=xy/(2t)$, we have
\begin{eqnarray}\label{eq-J}
{\rm J}&=&\frac{\kappa}{x^2}\left[-\frac{x^2}{t}+ 2x \frac{\partial_x E_\kappa\Big(\frac{x}{\sqrt{2t}},\frac{x}{\sqrt{2t}}\Big)}{  E_\kappa\Big(\frac{x}{\sqrt{2t}},\frac{x}{\sqrt{2t}}\Big)} + \log\frac{ E_\kappa\Big(-\frac{x}{\sqrt{2t}},\frac{x}{\sqrt{2t}}\Big)}{  E_\kappa\Big(\frac{x}{\sqrt{2t}},\frac{x}{\sqrt{2t}}\Big)}\right]\cr
&=&\frac{\kappa}{x^2}\left[-\frac{x^2}{t} + 2a \frac{\int_{-1}^1 sg(s) e^{as}\,\d s}{\int_{-1}^1 g(s) e^{as}\,\d s  } + \log\frac{\int_{-1}^1 g(s) e^{-as}\,\d s}{  \int_{-1}^1 g(s) e^{as}\,\d s}\right].
\end{eqnarray}

Set
$$f(a)=2a \frac{\int_{-1}^1 sg(s) e^{as}\,\d s}{\int_{-1}^1 g(s) e^{as}\,\d s  } + \log\frac{\int_{-1}^1 g(s) e^{-as}\,\d s}{  \int_{-1}^1 g(s) e^{as}\,\d s},\quad a\in\R.$$
We \textbf{claim} that $$f(a)\geq 0, \quad a\in\R.$$

Now we begin to prove the claim. Indeed, $f(0)=0$, and
\begin{eqnarray*}
f'(a)&=&2a \bigg[ \frac{\int_{-1}^1 s^2g(s) e^{as}\,\d s}{\int_{-1}^1 g(s) e^{as}\,\d s} -\frac{\big(\int_{-1}^1 sg(s) e^{as}\,\d s\big)^2}{\big(\int_{-1}^1 g(s) e^{as}\,\d s\big)^2} \bigg]\\
&&+\frac{\int_{-1}^1 sg(s) e^{as}\,\d s}{\int_{-1}^1 g(s) e^{as}\,\d s}-\frac{\int_{-1}^1 sg(s) e^{-as}\,\d s}{\int_{-1}^1 g(s) e^{-as}\,\d s}.
\end{eqnarray*}
Applying the Cauchy--Schwarz inequality, we have
\begin{eqnarray}\label{proof-J}
\frac{\int_{-1}^1 s^2g(s) e^{as}\,\d s}{\int_{-1}^1 g(s) e^{as}\,\d s} -\frac{\big(\int_{-1}^1 sg(s) e^{as}\,\d s\big)^2}{\big(\int_{-1}^1 g(s) e^{as}\,\d s\big)^2}\geq0.
\end{eqnarray}
For every $a\in\R$, set
$$h(a)=\Big(\int_{-1}^1 sg(s) e^{as}\,\d s\Big)\Big(\int_{-1}^1 g(s) e^{-as}\,\d s\Big)  -  \Big(\int_{-1}^1 g(s) e^{as}\,\d s\Big)
\Big(\int_{-1}^1 sg(s) e^{-as}\,\d s\Big).$$
Then $h(0)=0$, and
\begin{eqnarray*}
h'(a)&=&\Big(\int_{-1}^1 s^2g(s) e^{as}\,\d s\Big)\Big(\int_{-1}^1 g(s) e^{-as}\,\d s\Big) -2\Big(\int_{-1}^1 sg(s) e^{as}\,\d s\Big)\Big(\int_{-1}^1 sg(s) e^{-as}\,\d s\Big)\\
&&+\Big(\int_{-1}^1 g(s) e^{as}\,\d s\Big)\Big(\int_{-1}^1 s^2g(s) e^{-as}\,\d s\Big)\\
&\geq&AD-2\sqrt{ABCD}+BC=\big(\sqrt{AD}-\sqrt{BC} \big)^2\geq0,
\end{eqnarray*}
where we applied the Cauchy--Schwarz inequality twice and set
\begin{eqnarray*}
A&=&\int_{-1}^1 g(s) e^{as}\,\d s,\quad\,\,\,\,\,\, B=\int_{-1}^1 g(s) e^{-as}\,\d s,\\
C&=&\int_{-1}^1 s^2g(s) e^{as}\,\d s,\quad D=\int_{-1}^1 s^2g(s) e^{-as}\,\d s.
\end{eqnarray*}
Hence $a\mapsto h(a)$ is increasing in $\R$.

(1) Suppose $a\geq0$. Then, by \eqref{proof-J},
\begin{eqnarray*}
f'(a)&\geq&\frac{\int_{-1}^1 sg(s) e^{as}\,\d s}{\int_{-1}^1 g(s) e^{as}\,\d s} -  \frac{\int_{-1}^1 sg(s) e^{-as}\,\d s}{\int_{-1}^1 g(s) e^{-as}\,\d s}\\
&=&\frac{h(a)}{\big(\int_{-1}^1 g(s) e^{as}\,\d s\big)\big(\int_{-1}^1 g(s) e^{-as}\,\d s\big)}.
\end{eqnarray*}
Since $a\mapsto h(a)$ is increasing in $[0,\infty)$, we have $h(a)\geq h(0)=0$, $a\geq0$. Hence $f'(a)\geq0$, $a\geq0$.

(2) Suppose $a\leq0$. Then, by \eqref{proof-J},
\begin{eqnarray*}
f'(a)&\leq&\frac{\int_{-1}^1 sg(s) e^{as}\,\d s}{\int_{-1}^1 g(s) e^{as}\,\d s} -  \frac{\int_{-1}^1 sg(s) e^{-as}\,\d s}{\int_{-1}^1 g(s) e^{-as}\,\d s}\\
&=&\frac{h(a)}{\big(\int_{-1}^1 g(s) e^{as}\,\d s\big)\big(\int_{-1}^1 g(s) e^{-as}\,\d s\big)}.
\end{eqnarray*}
Since $a\mapsto h(a)$ is increasing in $(-\infty,0]$, we have $h(a)\leq h(0)=0$, $a\leq0$. Hence $f'(a)\leq0$, $a\leq0$.

Combining (1) and (2), we see that $f'(a)\geq0$ for every $a\geq0$, and $f'(a)\leq0$ for every $a\leq0$. Thus, $f(a)\geq f(0)=0$ for every $a\in\R$, which completes the proof of the claim.

Thus, by the claim and \eqref{eq-J}, we have
\begin{eqnarray}\label{proof-thm-1-2}
{\rm J}\geq -\frac{\kappa}{t}.
\end{eqnarray}

Combining \eqref{thm-1-proof-1}, \eqref{proof-thm-1-1} and \eqref{proof-thm-1-2}, we finally obtain
$$-\Delta_\kappa\big(\log p_t(\cdot,y)\big)(x)\leq\frac{d+2\lambda_\kappa}{2t},$$
which is \eqref{LY-1}.
\end{proof}

\section{Proofs of Theorem \ref{main-thm-2} and Corollary \ref{harnack}}\hskip\parindent
For $\psi\in C^1(\R)$ and $a,b\in\R$, let
$$\pi_\psi(a,b):=\psi(a)-\psi(b)-\psi'(b)(a-b).$$
In order to prove Theorem \ref{main-thm-2}, we need the following lemma which is motivated by \cite[Lemma 4.4]{GLR2018} in the particular case when $\psi(t)=|t|^p$ with $p>1$ for any $t\in\R$ and see also \cite[Lemma 2.1]{WZ2021} for the general pure jump case.
\begin{lemma}\label{chain-rule}
Let $I\subseteq\R$ be an interval, $\psi\in C^2(I)$ and $f\in C^2(\R^d,I)$. Then,
\begin{eqnarray*}
\Delta_\kappa\psi(f)=\psi'(f)\Delta_\kappa f+\psi''(f)|\nabla f|^2+\Pi_\psi(f),
\end{eqnarray*}
where
$$\Pi_\psi(f)(x):=2\sum_{\alpha\in\mathfrak{R}_+}\kappa_\alpha\frac{\pi_\psi\big(f(r_\alpha x),f(x)\big)}{\langle\alpha,x\rangle^2},\quad x\in\R^d.$$
In addition, if $f$ is $G$-invariant, i.e., $f(g x)=f(x)$ for every $g\in G$ and every $x\in\R^d$, then
$$\Delta_\kappa\psi(f)=\psi'(f)\Delta f+\psi''(f)|\nabla f|^2.$$
\end{lemma}
\begin{proof}
We only need to prove the first assertion, which is by direct calculation. For every $x\in\R^d$,
\begin{eqnarray*}
\Delta_\kappa\psi(f)(x)&=&\Delta\psi(f)(x)+2\sum_{\alpha\in\mathfrak{R}_+}\kappa_\alpha
\bigg(\frac{\langle\alpha,\nabla\psi\big(f(x)\big)\rangle}{\langle\alpha,x\rangle}+\frac{\psi\big(f(r_\alpha x)\big)-\psi\big(f(x)\big)}{\langle\alpha,x\rangle^2}\bigg)\\
&=&\psi'\big(f(x)\big)\Delta f(x)+\psi''\big(f(x)\big)|\nabla f|^2(x)\\
&&+2\sum_{\alpha\in\mathfrak{R}_+}\kappa_\alpha
\bigg(\psi'\big(f(x)\big)\frac{\langle\alpha,\nabla f(x) \rangle}{\langle\alpha,x\rangle}
+\psi'\big(f(x)\big)\frac{f(r_\alpha x)-f(x)}{\langle\alpha,x\rangle^2}\\
&&+\frac{\psi\big(f(r_\alpha x)\big)-\psi\big(f(x)\big)-\psi'\big(f(x)\big)[f(r_\alpha x)-f(x)]}{\langle\alpha,x\rangle^2}
\bigg)\\
&=&\psi'\big(f(x)\big)\Delta_\kappa f(x)+\psi''\big(f(x)\big)|\nabla f|^2(x)+2\sum_{\alpha\in\mathfrak{R}_+}\kappa_\alpha\frac{\pi_\psi\big(f(r_\alpha x),f(x)\big)}{\langle\alpha,x\rangle^2},
\end{eqnarray*}
where we applied the chain rule for the Laplacian $\Delta$ in the second equality.
\end{proof}

It is known that $p_t(x,y)$ is an integral kernel of $P_t$ and a $C^\infty$ function of all variables $x,y\in\R^d$ and $t>0$, which satisfies that, for every $x,y\in\R^d$ and $t>0$, $p_t(x,y)=p_t(y,x)>0$, $\partial_t p_t(x,y)=\Delta_\kappa\big(  p_t(\cdot,y)\big)(x)$,  and
\begin{eqnarray*}\label{kernel-bd}
p_t(x,y)\leq \frac{1}{c_\kappa(2t)^{d/2+\lambda_\kappa}}\exp\Big(-\frac{\delta(x,y)^2}{4t}\Big),
\end{eqnarray*}
where $c_\kappa=\int_{\R^d}e^{-|x|^2/2}\,\mu_\kappa(\d x)$ as above and $\delta(x,y):=\min_{g\in G}|x-gy|$. Refer to  \cite{Rosler2003,ADH2019} for more details.

Now we start to prove Theorem \ref{main-thm-2}.
\begin{proof}[Proof of Theorem \ref{main-thm-2}]
From Theorem \ref{main-thm-1} and the first assertion, it is clear that \eqref{thm-2-2} holds. Moreover, it is easy to see that \eqref{thm-2-1} implies \eqref{thm-2-0}.

Let $u$ be a positive $C^2$ solution to \eqref{dunkl-heat-equ} and denote also $u_t=u(t,\cdot)$. Applying Lemma \ref{chain-rule} with $\psi(t)=\log t$ for every $t>0$, we obtain
\begin{eqnarray}\label{main-thm-2-1}
\Delta_\kappa\log u_t=\frac{\Delta_\kappa u_t}{u_t}-|\nabla\log u_t|^2+\Pi_{\log}(u_t).
\end{eqnarray}
Then
\begin{eqnarray*}
\frac{|\nabla u(t,\cdot)(x)|^2}{u(t,x)^2}-\frac{\partial_t u(t,x)}{u(t,x)}&=&\Pi_{\log}\big(u(t,\cdot)\big)(x)-\Delta_\kappa\big(\log u(t,\cdot)\big)(x)\\
&=&2\sum_{\alpha\in\mathfrak{R}_+}\frac{\kappa_\alpha}{\langle\alpha,x\rangle^2}\Big(\log\frac{u(t,r_\alpha x)}{u(t,x)} - \frac{u(t,r_\alpha x)}{u(t,x)} +1\Big)\\
&&-\Delta_\kappa\big(\log u(t,\cdot)\big)(x).
\end{eqnarray*}
Let $\xi(t)=\log t -t+1$, $t>0$. It is easy to see that $\xi(t)\leq \xi(1)=0$ for all $t>0$, which implies that $\Pi_{\log}\big(u(t,\cdot)\big)(x)\leq 0$ for all $t>0$ and $x\in\R^d$. Together with \eqref{thm-2-1}, we prove \eqref{thm-2-1+}.

Hence, it remains to prove that \eqref{thm-2-0} implies \eqref{thm-2-1}.

Let $u(0,\cdot)=f$. By the standard approximation as in the proof of \cite[Theorem 1]{YZ2020}, we may assume that $f$ is compactly supported. Then $u_t=P_tf$ by uniqueness. Since $u(t,x)=P_tf(x)=\int_{\R^d}f(y)p_t(x,y)\,\mu_\kappa(\d y)$, we have
\begin{eqnarray}\label{proof-thm-2-1}
&&\partial_t u(t,x)+\beta(t,x)u(t,x)\cr
&=&\int_{\R^d}f(y)\partial_t p_t(x,y)\,\mu(\d y) + \beta(t,x)\int_{\R^d}f(y)p_t(x,y)\,\mu_\kappa(\d y)\cr
&=&\int_{\R^d}f(y)\big[\Delta_\kappa p_t(\cdot,y)(x)+\beta(t,x) p_t(x,y)\big]\,\mu_\kappa(\d y)\cr
&=&\int_{\R^d}f(y)\big[p_t(x,y)\Delta_\kappa\big(\log p_t(\cdot,y)\big)(x)+p_t(x,y)|\nabla\log p_t(\cdot,y)(x)|^2\cr
&&+\beta(t,x) p_t(x,y)-p_t(x,y)\Pi_{\log}\big(p_t(\cdot,y)\big)(x)\big]\,\mu_\kappa(\d y)\cr
&\geq&\int_{\R^d}f(y)\big[p_t(x,y)|\nabla\log p_t(\cdot,y)(x)|^2-p_t(x,y)\Pi_{\log}\big(p_t(\cdot,y)\big)(x)\big]\,\mu_\kappa(\d y),
\end{eqnarray}
where we applied the dominated convergence theorem in the first equality (see also \cite{ADH2019}), \eqref{main-thm-2-1} in the third equality and \eqref{thm-2-0} in the last inequality.

It is easy to see that, by the Cauchy--Schwarz inequality,
\begin{eqnarray*}
|\nabla P_tf(x)|^2&\leq&\Big(\int_{\R^d}|\nabla p_t(\cdot,y)(x)|f(y)\,\mu_\kappa(\d y)\Big)^2\\
&\leq&\Big(\int_{\R^d}\frac{|\nabla p_t(\cdot,y)(x)|^2}{p_t(x,y)}f(y)\,\mu_\kappa(\d y)\Big)\Big(\int_{\R^d}p_t(x,y)f(y)\,\mu_\kappa(\d y)\Big),
\end{eqnarray*}
which implies that
\begin{eqnarray}\label{proof-thm-2-2}
\int_{\R^d}f(y) |\nabla\log p_t(\cdot,y)(x)|^2p_t(x,y)\,\mu_\kappa(\d y) \geq P_tf(x) |\nabla \log P_tf(x)|^2.
\end{eqnarray}

Now let
$${\rm I}=-\int_{\R^d}\Pi_{\log}\big(p_t(\cdot,y)\big)(x)p_t(x,y)f(y)\,\mu_\kappa(\d y),$$
and
$${\rm II}=-P_tf(x)\Pi_{\log}(  P_t f)(x).$$
Then
\begin{eqnarray*}
{\rm I}-{\rm II}&=&2\sum_{\alpha\in\mathfrak{R}_+}\frac{\kappa_\alpha}{\langle \alpha,x\rangle^2}\int_{\R^d} \Big[-\log\frac{p_t(r_\alpha x,y)}{p_t(x,y)}+\frac{p_t(r_\alpha x,y)}{p_t(x,y)}-1 \Big]p_t(x,y)f(y)\,\mu_\kappa(\d y)\\
&&- 2\sum_{\alpha\in\mathfrak{R}_+}\frac{\kappa_\alpha}{\langle \alpha,x\rangle^2}\int_{\R^d} \Big[-\log\frac{P_tf(r_\alpha x)}{P_tf(x)}+\frac{P_tf(r_\alpha x)}{P_tf(x)}-1 \Big]p_t(x,y)f(y)\,\mu_\kappa(\d y)\\
&=&2\sum_{\alpha\in\mathfrak{R}_+}\frac{\kappa_\alpha}{\langle \alpha,x\rangle^2}\int_{\R^d} \left[\eta\Big(\frac{p_t(r_\alpha x,y)}{p_t(x,y)}\Big) -  \eta\Big(\frac{P_tf(r_\alpha x)}{P_tf(x)}\Big) \right]p_t(x,y)f(y)\,\mu_\kappa(\d y),
\end{eqnarray*}
where $\eta(t)=t-\log t-1$, $t\in\R$. Since $t\mapsto \eta(t)$ is convex in $(0,\infty)$ and $ \frac{\d}{\d t}\eta(t)=(t-1)/t$, we have
$$ \eta(t)- \eta(s)\geq\frac{s-1}{s}(t-s), \quad s,t>0.$$
Hence
\begin{eqnarray*}
&&{\rm I}-{\rm II}\\
&\geq& 2\sum_{\alpha\in\mathfrak{R}_+}\frac{\kappa_\alpha}{\langle \alpha,x\rangle^2}
\int_{\R^d} \frac{P_tf(r_\alpha x)-P_tf(x)}{P_tf(r_\alpha x)}\Big[\frac{p_t(r_\alpha x,y)}{p_t(x,y)}-\frac{P_tf(r_\alpha x)}{P_tf(x)}\Big]
p_t(x,y)f(y)\,\mu_\kappa(\d y)\\
&=&2\sum_{\alpha\in\mathfrak{R}_+}\frac{\kappa_\alpha}{\langle \alpha,x\rangle^2}\bigg(
\int_{\R^d} \frac{P_tf(r_\alpha x)-P_tf(x)}{P_tf(r_\alpha x)}p_t(r_\alpha x,y)f(y)\,\mu_\kappa(\d y)\\
&&-\int_{\R^d}\frac{P_tf(r_\alpha x)-P_tf(x)}{P_tf(x)}p_t(x,y)
f(y)\,\mu_\kappa(\d y)
\bigg)\\
&=&2\sum_{\alpha\in\mathfrak{R}_+}\frac{\kappa_\alpha}{\langle \alpha,x\rangle^2}\Big(\big[P_tf(r_\alpha x)-P_tf(x)\big]- \big[P_tf(r_\alpha x)-P_tf(x)\big]\Big)\\
&=&0,
\end{eqnarray*}
which means that
\begin{eqnarray}\label{proof-thm-2-3}
\int_{\R^d}\Pi_{\log}\big(p_t(\cdot,y)\big)(x)p_t(x,y)f(y)\,\mu_\kappa(\d y)\leq P_tf(x)\Pi_{\log}( P_t f)(x).
\end{eqnarray}

Combining \eqref{proof-thm-2-1}, \eqref{proof-thm-2-2} and \eqref{proof-thm-2-3}, we have
\begin{eqnarray*}\label{proof-thm-2-4}
\partial_t u(t,x)+\beta(t,x)u(t,x)&\geq& P_tf(x)|\nabla\log P_t f(x)|^2-P_tf(x)\Pi_{\log}(P_tf)(x)\\
&=&u(t,x)|\nabla u(t,\cdot)(x)|^2-u(t,x)\Pi_{\log}\big(u(t,\cdot)\big)(x),
\end{eqnarray*}
which is equivalent to
$$-\Delta_\kappa\big(\log u(t,\cdot)\big)(x)\leq \beta(t,x)$$
by \eqref{main-thm-2-1} and \eqref{thm-2-0}.
\end{proof}

The proof of the Harnack inequality is standard. Now we present it here for the sake of completeness.
\begin{proof}[Prof of Corollary \ref{harnack}]
By Theorem \ref{main-thm-2}, it is clear that
\begin{eqnarray}\label{pf-harnack-1}
\frac{|\nabla u(t,\cdot)(x)|^2}{u(t,x)^2}-\frac{\partial_t u(t,x)}{u(t,x)}\leq  \frac{d+2\lambda_\kappa}{2t},\quad (t,x)\in (0,T)\times \R^d.
\end{eqnarray}
Let $0<s<t<T$, $x,y\in\R^d$, and
$$\gamma_\tau=\big(t+\tau(s-t),y+\tau(x-y)\big),\quad \tau\in[0,1],$$
be the straight line starting from $(t,y)$ to $(s,x)$.
Consider the function
$$\phi(\tau):=\log u(\gamma_{\tau}),\quad \tau\in[0,1].$$
 Then
\begin{eqnarray*}
&&\log \frac{u(s,x)}{u(t,y)}=\phi(1)-\phi(0)=\int_0^1 \phi'(\tau)\,\d\tau\\
&=&\int_0^1\left[\Big\langle\frac{\nabla_x u(\gamma_\tau)}{u(\gamma_\tau)}, x-y\Big\rangle-(t-s)\frac{\partial_t u(\gamma_\tau)}{u(\gamma_\tau)}\right]\,\d\tau\\
&\leq&\int_0^1|x-y|\frac{|\nabla_x u(\gamma_\tau)|}{u(\gamma_\tau)}\,\d\tau - (t-s)\int_0^1 \frac{|\nabla_x u(\gamma_\tau)|^2}{u(\gamma_\tau)^2}\,\d\tau
 + \int_0^1\frac{(t-s)(d+2\lambda_\kappa)}{2\big(t+\tau(s-t)\big)}\,\d\tau\\
&\leq&|x-y|\Big(\int_0^1\frac{|\nabla_x u(\gamma_\tau)|^2}{u(\gamma_\tau)^2}\,\d\tau\Big)^{1/2}- (t-s)\int_0^1 \frac{|\nabla_x u(\gamma_\tau)|^2}{u(\gamma_\tau)^2}\,\d\tau +\log\Big(\frac{t}{s}\Big)^{\lambda_\kappa+d/2}\\
&\leq& \frac{|x-y|^2}{4(t-s)}+\log\Big(\frac{t}{s}\Big)^{\lambda_\kappa+d/2},
\end{eqnarray*}
where we applied \eqref{pf-harnack-1} in the first inequality, the Cauchy--Schwarz inequality in the second one and Young's inequality in the last one. Thus, we immediately have
$$u(s,x)\leq u(t,y)\Big(\frac{t}{s}\Big)^{\lambda_\kappa+d/2}\exp\Big(\frac{|x-y|^2}{4(t-s)}\Big).$$
\end{proof}

There is a remark on the geometric interpretation of \eqref{pf-harnack-1} as \cite[Corollary 1.4]{HeTop2013}.
\begin{remark}
Let $\varrho_t(x)=p_t(x,0)$. Then, under the same assumption of Corollary \ref{harnack},  for every $t\in (0,T)$, $\R^d\ni x\mapsto \log(\frac{u_t}{\varrho_t})(x)$ is convex. Indeed, for every $t>0,\,x,y\in\R^d$ with $x=(x_1,\cdots,x_d),y=(y_1,\cdots,y_d)$, letting $q_t(x,y)=p_t(x,y)/\varrho_t(x)$,  by \eqref{1d-kernel}, \eqref{product-kernel} and the fact that $E_{\kappa_j}(x_j,0)=1$ for every $j=1,\cdots,d$  (see e.g. \cite[page 2366]{ADH2019}),  we have
$$q_t(x,y)=\exp\Big(-\frac{|y|^2}{4t}\Big)\prod_{j=1}^d c_j\int_{-1}^{1}g_j(s)e^{\frac{sx_jy_j}{2t}}\,\d s,$$
where $c_j:=\frac{\Gamma(\kappa_i+1/2)}{\sqrt{\pi} \Gamma(\kappa_i)}$ and $g_j(s):=(1-s)^{\kappa_j-1}(1+s)^{\kappa_j}$, $j=1,\cdots,d$. Then it is easy to see that $\partial_{x_i x_j}^2\log q_t(x,y)=0$, $1\leq i\neq j\leq d$, and for each $j=1,\cdots,d$,
\begin{eqnarray*}
\partial_{x_j x_j}^2\log q_t(x,y)&=&\frac{y_j^2}{4t^2}\frac{\int_{-1}^1 s^2g_j(s)e^{\frac{sx_jy_j}{2t}}\,\d s}{\int_{-1}^1 g_j(s)e^{\frac{sx_jy_j}{2t}}\,\d s} -\frac{y_j}{2t}\frac{\big(\int_{-1}^1 s g_j(s)e^{\frac{sx_jy_j}{2t}}\,\d s\big)^2}{\big(\int_{-1}^1 g_j(s)e^{\frac{sx_jy_j}{2t}}\,\d s\big)^2}\\
&\geq&0,
\end{eqnarray*}
by the Cauchy--Schwarz inequality. Hence $x\mapsto \log q_t(x,y)$ is convex. Thus, for every $\theta\in(0,1)$ and every $z_1,z_2\in\R^d$, by H\"{o}lder's inequality, we deduce that
\begin{eqnarray*}
&&\Big(\frac{u_t}{\varrho_t}\Big)\big(\theta z_1+(1-\theta)z_2\big)\\
&=&\int_{\R^d}q_t\big(\theta z_1+(1-\theta)z_2,y\big)u(0,y)\,\mu_\kappa(\d y)\\
&\leq&\int_{\R^d} q_t(z_1,y)^\theta q_t(z_2,y)^{1-\theta}u(0,y)\,\mu_\kappa(\d y)\\
&\leq&\Big(\int_{\R^d} q_t(z_1,y)u(0,y)\,\mu_\kappa(\d y)\Big)^\theta\Big(\int_{\R^d} q_t(z_2,y)u(0,y)\,\mu_\kappa(\d y)\Big)^{1-\theta}\\
&=&\Big[\Big(\frac{u_t}{\varrho_t}\Big)(z_1)\Big]^\theta\Big[\Big(\frac{u_t}{\varrho_t}\Big)(z_2)\Big]^{1-\theta},
\end{eqnarray*}
which implies that $\log(\frac{u_t}{\varrho_t})(\cdot)$ is convex.
\end{remark}

\subsection*{Acknowledgment}\hskip\parindent
The first named author would like to thank Dr. Niushan Gao for helpful discussions and acknowledge the Department of Mathematics and the Faculty of Science at Ryerson University for financial support and the financial support from the National Natural Science Foundation of China (Grant No. 11831014). The second named author would like to acknowledge the financial support from Qing Lan Project of Jiangsu.

\end{document}